\documentclass[preprint,10pt]{elsarticle}

\usepackage{color}
\usepackage{amssymb}
 \usepackage{amsthm}
\usepackage{multicol}
\usepackage{color}
\usepackage{amsmath}

\usepackage{array}
\usepackage{amsfonts}
\usepackage{bm}
\usepackage{algorithm}
\usepackage{algorithmic}
\usepackage{amscd}
\usepackage{mathrsfs}
\usepackage[figuresright]{rotating}

\usepackage[colorlinks, linkcolor=black,dvipdfm]{hyperref}

\usepackage{lineno}
\modulolinenumbers[5]

\newtheorem{lemma}{Lemma}[section]

\newtheorem{theorem}{Theorem}[section]
\newtheorem{example}{Example}[section]

\newtheorem{remark}{\em Remark}[section]

\journal{Computers and Mathematics with Applications}

\begin{document}
\begin{frontmatter}
\title{Numerical schemes of the time tempered fractional Feynman-Kac equation
}
\author[mymainaddress]{W.H. Deng}
\ead{dengwh@lzu.edu.cn}
\author[mymainaddress]{Z.J. Zhang\corref{mycorrespondingauthor}}
\cortext[mycorrespondingauthor]{Corresponding author.}
\ead{zhjzhang14@lzu.edu.cn}
\address[mymainaddress]{School of Mathematics and Statistics, Gansu Key Laboratory of Applied Mathematics and Complex Systems, Lanzhou University, Lanzhou 730000, P.R. China}

\begin{abstract}
This paper focuses on providing the computation methods for the backward time tempered fractional Feynman-Kac equation, being one of the models recently proposed in [Wu, Deng, and Barkai, Phys. Rev. E, 84 (2016) 032151].
The discretization for  the  tempered  fractional substantial derivative is derived, and the corresponding finite difference and finite element schemes are designed with well established stability and convergence. The performed numerical experiments show the effectiveness of the presented schemes.
\end{abstract}


\end{frontmatter}

\section{Introduction} \label{introduction}

Anomalous diffusion  is one of the most ubiquitous phenomena in nature, being detected in almost all the scientific fields.
In particular, based on the continuous time random walk (CTRW) model with the power law waiting time distribution having divergent first moment and/or the power law jump length distribution having divergent second moment,
 the corresponding fractional subdiffusive and/or superdiffusive  equations are derived \cite{Metzler:00,Schneider:89}
 (for their numerical methods, see, e.g., \cite{Deng:13,Ervin:07,Gao:16,Meerschaert:13,Pang:12,Tian:15,Zhang:11}).
But most of the time, the more practical choice for modelling the motion of the particles should use the distributions of the waiting time and jump length with semi-heavy tails, due to the finite life span
and/or the bounded physical space of the biological particles.
These can be realized by tempering the distributions with heavy tails, e.g., truncating the heavy tail of the power law distribution. What tempering does is to introduce a scale. Exponentially tempering the power-law distributions
seems to be the most popular choice \cite{Cartea:07,Meerschaert:09}, since it has both the mathematical and technique advantages \cite{Baenmer:10,Sabzikar:15}; and the probability densities of the tempered stable process solve the tempered fractional diffusion equation.

There are many physical quantities that can be used to characterize the motion features of a Brownian particle, e.g., the Brownian functionals \cite{Majumdar:05}. With the rapid development of the research of anomalous diffusion, the functionals of the non-Brownian particle naturally attract the interests of scientists \cite{Carmi:10,Turgeman:09,Xiaochao:16}.
The functional is defined as $A=\int_0^tU(x_0(\tau))d\tau$,
 where $U(x_0)$ is a prescribed function and $x_0(t)$ is a random process generated by a non-Brownian particle.
The choice of $U(x_0)$ depends on the concrete applications. For example, in the kinetic studies of chemical reactions that take place exclusively in some given domain \cite{Agmon:84,Bar:98,Carmi:10}, we take $U(x_0)=1$ in this particular domain and zero otherwise. For inhomogeneous discorder dispersive systems, $U(x_0)$ is taken as $x_0$ or $x_0^2$ \cite{Carmi:10}, etc.
In particular, based on the sub-diffusive CTRW, a widely investigated process being continually used
to characterize the motion of particles in disordered systems \cite{scher:75},
the fractional Feynman-Kac equation is derived in \cite{Carmi:10,Turgeman:09}.
 More recently, taking the tempered power law function as the waiting time and/or jump length distribution(s) in the CTRW model,
 the tempered Feynman-Kac equations are also derived in \cite{Xiaochao:16}, which govern the distribution of the functionals of the tempered anomalous diffusion.

In this paper, we consider the numerical schemes of the backward tempered fractional Feynman-Kac equation proposed in \cite[eq. 15]{Xiaochao:16}, i.e.,
\begin{eqnarray}\label{eq-mode2}
&&\frac{\partial }{\partial t}G_{x_0}(p,t)=\left[\lambda^{\gamma} D_t^{1-\gamma,\lambda}-\lambda\right]\left[G_{x_0}(p,t)-e^{-pU(x_0)t}\right]\nonumber\\
&&~~~~~~~~~~~~~~~~~~ -pU(x_0)G_{x_0}(p,t)+K_{\gamma} D_t^{1-\gamma,\lambda}\frac{\partial ^2}{\partial x_0^2}G_{x_0}(p,t),
\end{eqnarray}
where $G_{x_0}(p,t)=\int_0^{\infty}e^{-pA}G_{x_0}(A,t)dA,\,{\rm Re}(p)>0$ is the Lapace transform of $G_{x_0}(A,t)$,
 being the probability density function (PDF) of $A=\int_0^t U([x_0(\tau)]d\tau$ at time $t$ for a process starting at $x_0(0)$; $U(x_0)\ge0$ (then $ A\ge 0$) is a prescribed function;
the diffusion coefficient $K_\gamma>0$ and the tempered parameter  $\lambda\ge 0$ are constants, $\gamma\in(0,1)$;
and the operator $D_t^{1-\gamma,\lambda}$ on the right-hand side of (\ref{eq-mode2}) is the  tempered  fractional substantial derivative, being defined as,
\begin{eqnarray}\label{eq-mode2-1}
&&D_t^{1-\gamma,\lambda}G_{x_0}(p,t)=\nonumber\\
&&~~~~\frac{1}{\Gamma(\gamma)}\left[\lambda+pU(x_0)+\frac{\partial }{\partial t}\right]\int_0^t\frac{e^{-(t-\xi)(\lambda+pU(x_0))}}
{(t-\xi)^{1-\gamma}}G_{x_0}(p,\xi) d\xi,
\end{eqnarray}
which is a time-space coupled operator.
Moreover, if $\lambda=0$ and $p\not=0$,  Eq. (\ref{eq-mode2}) reduces to the standard backward fractional
Feynman-Kac equation studied in \cite{Chen:15,Deng:15}; if $\lambda\not=0$ and $p=0$, Eq. (\ref{eq-mode2})
reduces to the time tempered fractional equation originally proposed in \cite{Meerschaert:09} and numerically solved in \cite{Hanert:14}; if $\lambda=p=0$, Eq. (\ref{eq-mode2}) reduces to the anomalous subdiffusive equation
discussed in, e.g., \cite{Chen:09,Li:09,Zhang:11,Zhuang:08}.

The outline of this paper is as follows. In Sec. \ref{sec2}, we derive an equivalent form of (\ref{eq-mode2})
and give the discretization of the tempered fractional substantial derivative. The finite difference scheme and the related theoretical analysis for solving the model (\ref{eq-mode2}) are provided in Sec. \ref{sec3}. Then the finite element approximation and its numerical analysis for the model (\ref{eq-mode2}) are presented in Sec. \ref{sec4}.
Numerical simulation results are reported in Sec. \ref{sec5} to confirm the effectiveness of the given scheme.
And we conclude the paper in  the last section.

\section{Equivalent form of (\ref{eq-mode2}) and discretization of the tempered fractional substantial derivative}\label{sec2}
This section focuses on deriving the equivalent form of (\ref{eq-mode2}) and the discretization of the tempered fractional substantial derivative. Under the assumption that the solution of (\ref{eq-mode2}) is sufficiently regular, we firstly derive its equivalent form.  For $0<\gamma<1$, assuming that $v(t)$ is regular enough,  then ${}_0 I_t^{\gamma}v(t)=\frac{1}{\Gamma(\gamma)}\int_0^t{(t-\xi)^{\gamma-1}}v(\xi)d\xi$, ${}_0 D_t^{\gamma}v(t)=\frac{\partial}{\partial t}\left({}_0I_t^{1-\gamma}v(t)\right)$, and ${}_0^CD_t^{\gamma}v(t)
 ={}_0I_t^{1-\gamma}\left(\frac{\partial v(t)}{\partial t}\right)$ are, respectively, defined as the Riemann-Liouville integral, the Riemann-Liouville derivative, and the Caputo derivative \cite{Podlubny:99,Zhou:14}; and there hold that
\begin{eqnarray} \label{IdenttityOperator}
{}_0D_t^{\gamma}\,{}_0I_t^{\gamma}v(t)=v(t)
\end{eqnarray}
and
\begin{eqnarray} \label{IdenttityOperator2}
{}_0I_t^{\gamma}\,{}_0D_t^{\gamma }v(t)=v(t).
\end{eqnarray}
See  \ref{appendix1} for the proof of Eqs. (\ref{IdenttityOperator}) and (\ref{IdenttityOperator2}).
And it is easy to check that Eq. (\ref{eq-mode2-1}) can be rewritten as
\begin{equation}\label{substantialdefinition}
\begin{array}{l}
D_t^{1-\gamma,\lambda}G_{x_0}(p,t)\\
=\left[\lambda+pU(x_0)+\frac{\partial }{\partial t}\right]
\left( e^{-(\lambda+pU(x_0))t}{}_0I_t^{\gamma}\left(e^{(\lambda+pU(x_0))t}G_{x_0}(p,t)\right)\right).
\end{array}
\end{equation}
In fact, we can further rewrite it as

\begin{lemma}\label{substantilemma1}
For $0<\gamma<1$, there holds
\begin{eqnarray} \label{eq6}
D_t^{1-\gamma,\lambda}G_{x_0}(p,t)=e^{-(\lambda+pU(x_0))t}{}_0D_t^{1-\gamma}\left(e^{(\lambda+pU(x_0))t}G_{x_0}(p,t)\right).
\end{eqnarray}
\end{lemma}
See  \ref{appendix2} for the proof of Eq. (\ref{eq6}).
By Lemmas \ref{substantilemma1} and Eq. (\ref{IdenttityOperator2}), if $G_{x_0}(p,t)$ with respect to $t$ lies in $C^1[0,T]$,
then model (\ref{eq-mode2}) has the following equivalent form
\begin{eqnarray} \label{eq7}
&&e^{-\left(\lambda+pU(x_0)\right)t}\,{}_0^CD_t^{\gamma}\left(e^{\left(\lambda+pU(x_0)\right)t}G_{x_0}(p,t)\right)-\lambda^{\gamma}G_{x_0}(p,t)\nonumber\\
&&=K_{\gamma} \frac{\partial ^2}{\partial x_0^2}G_{x_0}(p,t)+f(x_0,p,t),\label{modemode2}
\end{eqnarray}
where $f(x_0,p,t)=-\lambda^{\gamma} e^{-pU(x_0)t}+\lambda e^{-\left(\lambda+pU(x_0)\right)t}{}_0I_t^{1-\gamma}(e^{\lambda t})$.

Let's briefly introduce the deriving process of Eq. (\ref{eq7}). One can firstly rewrite model  (\ref{eq-mode2})  as
\begin{eqnarray*}
&&\left(\lambda+pU(x_0)+\frac{\partial}{\partial t} \right)G_{x_0}(p,t)\nonumber\\
&&=D_t^{1-\gamma,\lambda}\left(\lambda^{\gamma}+K_{\gamma} \frac{\partial ^2}{\partial x_0^2}\right)G_{x_0}(p,t)
-\left[\lambda^{\gamma} D_t^{1-\gamma,\lambda}-\lambda\right]e^{-pU(x_0)t}.
\end{eqnarray*}
Note that
\begin{eqnarray*}
\left(\lambda+pU(x_0)+\frac{\partial}{\partial t}\right)G_{x_0}(p,t)=e^{-(\lambda+pU(x_0))t}\frac{\partial}{\partial t}\left(e^{(\lambda+pU(x_0))t}G_{x_0}(p,t)\right).
\end{eqnarray*}
Then using Lemma \ref{substantilemma1}, it follows that
\begin{eqnarray}
&&\frac{\partial}{\partial t}\left(e^{(\lambda+pU(x_0))t}G_{x_0}(p,t)\right)\\
&& ={}_0D_t^{1-\gamma}\left(e^{(\lambda+pU(x_0))t}\left(\lambda^{\gamma}
+K_{\gamma} \frac{\partial ^2}{\partial x_0^2}\right)G_{x_0}(p,t)\right)\nonumber \\
&& -e^{(\lambda+pU(x_0))t}\left(\lambda^{\gamma} D_t^{1-\gamma,\lambda}e^{-pU(x_0)t}-\lambda e^{-pU(x_0)t}\right).\label{eqderivativeq}
\end{eqnarray}
Acting  ${}_0I_t^{(1-\gamma)}$ on  both sides of (\ref{eqderivativeq}), and taking $v(t)$ in (\ref{IdenttityOperator2}) as $e^{(\lambda+pU(x_0))t}G_{x_0}(p,t)$, it yields
\begin{eqnarray}
&&e^{-\left(\lambda+pU(x_0)\right)t}\,{}_0^CD_t^{\gamma}\left(e^{\left(\lambda+pU(x_0)\right)t}G_{x_0}(p,t)\right)-\lambda^{\gamma}G_{x_0}(p,t)\nonumber\\
&&~=K_{\gamma} \frac{\partial ^2}{\partial x_0^2}G_{x_0}(p,t)+f(x_0,p,t)
\end{eqnarray}
with
\begin{eqnarray}
&&f(x_0,p,t)=-e^{-(\lambda+pU)t}{}_0I_t^{1-\gamma}\left[e^{(\lambda+pU)t}\left(\lambda^{\gamma} D_t^{1-\gamma,\lambda}e^{-pUt}
-\lambda e^{-pUt}\right)\right]\nonumber\\
&&~~~~~~~~~~~~~=-\lambda^{\gamma} e^{-pUt}+\lambda e^{-\left(\lambda+pU\right)t}{}_0I_t^{1-\gamma}(e^{\lambda t}).
\end{eqnarray}
This process can also be similarly inversed. One can develop the numerical scheme of (\ref {eq-mode2}) based on the form (\ref{modemode2}); since for the given initial and boundary conditions:
\begin{eqnarray}
G_{x_0}(p,0)=\phi(p,x_0),\quad G_a(p,t)=\psi_l(p,t),\quad G_b(p,t)=\psi_r(p,t),
\end{eqnarray}
one can let
\begin{eqnarray}
&&G_{x_0}(p,t)=W_{x_0}(p,t)+\phi(p,x_0) e^{-(\lambda+pU(x_0))t}\nonumber\\
&&~~~~~~~~~~~~~~+\frac{(x_0-a)}{(b-a)}\left[\psi_r(p,t)e^{(\lambda+pU(b))t}-\phi_(p,b)\right]e^{-(\lambda+pU(x_0))t}\nonumber\\
&&~~~~~~~~~~~~~~+\frac{(b-x_0)}{(b-a)}\left[\psi_l(p,t)e^{(\lambda+pU(a))t}-\phi(p,a)\right]e^{-(\lambda+pU(x_0))t}\label{huanyuan}
\end{eqnarray}
to obtain a model with respect to $W_{x_0}(p,t)$ satisfying $W_{x_0}(p,0)=W_a(p,t)=W_b(p,t)=0$. So, without loss of generality, we assume that $\phi(p,x_0)=\psi_l(p,t)=\psi_r(p,t)=0$.

Now, based on Eq. (\ref{eq7}), we consider the discretization of its time dependent operator.
Let $0=t_0<\cdots<t_n<\cdots<t_N=T$ with $t_n=n\tau$ be the  time partition of $[0,T]$,  and denote
\begin{eqnarray*}
\mathscr{S}^{1+\gamma}(\mathbb{R})=\left\{v\bigg|v\in L^1(\mathbb{R});
 \int_{\mathbb{R}}\left(1+|\omega|\right)^{1+\gamma}|\mathscr{F}[{v(t)}](\omega)|d\omega<\infty\right\}.
\end{eqnarray*}
Here $\mathscr{F}[{v}(t)](\omega)=\int_{-\infty}^{\infty}e^{-i\omega t}v(t)dt$ denotes the Fourier transform of $v(t)$.
In the following, we assume that $|pU|$ is bounded.

Note that for $v(0)=0$, one has ${}_0^CD_t^{\gamma}v(t)={}_0D_t^{\gamma}v(t)$ and  the  Gr\"unwald  approximation \cite{Meerschaert:13,Podlubny:99}:
\begin{eqnarray}
{}_0D_t^{\gamma}v(t_n)=\frac{1}{\tau^{\gamma}}\sum_{k=0}^ng_k^{(\gamma)}v(t_{n-k})+\mathcal{O}(\tau), {~~\rm if~} \tilde {v}(t)\in \mathscr{S}^{1+\gamma}(\mathbb{R}),
\end{eqnarray}
where $\sum_{k=0}^{\infty}g_k^{(\gamma)}z^k=(1-z)^{\gamma}$, and $\tilde {v}(t)$ is the zero extension of $v(t)$ for $t\le 0$.
By Lemma \ref {substantilemma1}, it holds that
\begin{eqnarray}
e^{-\left(\lambda+pU(x_0)\right)t}\,{}_0^CD_t^{\gamma}\left(e^{\left(\lambda+pU(x_0)\right)t}G_{x_0}(p,t)\right)=D_t^{\gamma,\lambda}G_{x_0}(p,t);
\end{eqnarray}
and one may desire:
\begin{eqnarray} \label{eq16}
&& D_t^{\gamma,\lambda}G_{x_0}(p,t)\bigg|_{t=t_n} \nonumber\\
&& = \frac{e^{-(\lambda+pU(x_0))t_n}}{\tau^{\gamma}}\sum_{k=0}^ng_k^{(\gamma)}e^{(\lambda+pU(x_0))t_{n-k}}G_{x_0}(p,t_{n-k})
+\mathcal{O}(\tau)\nonumber\\
&& =\frac{1}{\tau^{\gamma}}\sum_{k=0}^ng_k^{(\gamma,\lambda)} e^{-pU(x_0)k\tau}G_{x_0}(p,t_{n-k})+\mathcal{O}(\tau),\label{approximationeq}
\end{eqnarray}
 where $g_k^{(\gamma,\lambda)}=e^{-k\lambda\tau}g_k^{(\gamma)}$ satisfying $\sum_{k=0}^{\infty}g_k^{(\gamma,\lambda)}z^k
 =\left(1-ze^{-\lambda \tau}\right)^{\gamma}$.
 In fact, currently $D_t^{\gamma,\lambda}G_{x_0}(p,t)$  can be regarded as the tempered fractional derivative given in \cite{Sabzikar:15}
 with the tempered parameter $\lambda+pU(x_0)$. Therefore,  the proof of $(\ref{approximationeq})$
 is similar to \cite[Theorem 2.4]{Meerschaert:13} and  \cite[Theorem 5.1]{Sabzikar:15}.  And for the properties of the coefficient $g_k^{(\gamma,\lambda)}$, it is easy to check that
\begin{lemma}\label{ceofficienteq}
For $0<\gamma<1$, it follows that
\begin{eqnarray}
&&g_0^{(\gamma,\lambda)}=1,\quad g_k^{(\gamma,\lambda)}=e^{-\lambda \tau}\left(1-\frac{\gamma+1}{k}\right)g_{k-1}^{(\gamma,\lambda)}\quad k\ge 1,\\
&&g_k^{(\gamma,\lambda)}<0,\quad \sum_{k=0}^ng_k^{(\gamma,\lambda)}>\left(1-e^{-\lambda\tau}\right)^\gamma.
\end{eqnarray}
\end{lemma}

\section {Finite difference scheme and the stability and convergence analysis}\label{sec3}
In this section, we derive the finite difference scheme of Eq. (\ref{eq7}) and present the detailed stability and convergence analysis.
\subsection{Derivation of the difference scheme}
Let $a=x_{0,0}<\cdots x_{0,m}<\cdots<x_{0,M}=b$ with $x_{0,m}=mh,\,h=\frac{b-a}{M}$ be the space partition.
And denote $G_m^n$ as the numerical approximation of $G_{x_{0,m}}(p,t_n),\,\Omega:=(a,b)$. For $G_{x_0}(p,\cdot)\in C^4(\overline{\Omega})$, the second order central difference formula for the spatial derivative is given as
\begin{eqnarray}\label{spacediscrete2}
\frac{\partial ^2}{\partial x_0^2}G_{x_{0}}(p,t)\Big|_{(x_{0,m},t_n)}\approx\frac{G_{x_{0,m-1}}(p,t_n)-2G_{x_{0,m}}(p,t_n)+G_{x_{0,m+1}}(p,t_n)}{h^2}
\end{eqnarray}
with the local truncation error $\mathcal{O}(h^2)$. And from (\ref{approximationeq}), one has
\begin{eqnarray}
&&\Big[e^{-(\lambda+pU(x_{0,m}))t}\,{}_0^CD_t^{\gamma}\left(e^{(\lambda+pU(x_{0,m}))t}G_{x_{0,m}}(p,t)\right)
-\lambda^{\gamma}G_{x_{0,m}}(p,t)\Big]\Big|_{t=t_n}\nonumber\\
&&\approx \frac{1}{\tau^{\gamma}} \sum_{k=0}^{n}g_k^{(\gamma,\lambda)}e^{-pU(x_{0,m})k\tau}G_{x_{0,m}}(p,t_n-k\tau)
-e^{-\gamma \lambda \tau}\lambda^\gamma G_{x_{0,m}}(p,t_n)\nonumber\\
&&\approx\frac{1}{\tau^{\gamma}}\sum_{k=0}^n d_k^{(\gamma,\lambda)}e^{-pU(x_{0,m})k\tau}G_{x_{0,m}}(p,t_n-k\tau) \label{timediscrete1}
\end{eqnarray}
with the local truncation error $\mathcal{O}(\tau) $, where $e^{-\gamma\lambda\tau}=1-\gamma\lambda \tau+\mathcal{O}(\tau^2)$ has been used,
 and $d_0^{(\gamma,\lambda)}=1-e^{-\gamma \lambda \tau}(\lambda\tau)^\gamma$, $ d_k^{(\gamma,\lambda)}=g_k^{(\gamma,\lambda)}$ for $k\ge 1$.

Denoting
\begin{eqnarray*}
&&{\rm M_k}={\rm diag}\left(e^{-pU(x_{0,1})k\tau},\cdots, e^{-pU(x_{0,m})k\tau},\cdots,e^{-pU(x_{0,M-1})k\tau}\right),\\
&& {\rm H}=\frac{1}{h^2}\,{\rm tridiag}\left(1,-2,1\right),~~ G_h^n=\left(G_1^n,\cdots,G_m^n,\cdots,G_{M-1}^n\right)^T,\\
&& F^n=\left(f({x_{0,1}},p,t^n),\cdots,f({x_{0,m}},p,t^n),\cdots,f(x_{0,M-1},p,t^n)\right)^T,
\end{eqnarray*}
and omitting the truncation error terms, the difference scheme of (\ref{modemode2}) can be given as
\begin{eqnarray}\label{differencescheme2}
&&\sum_{k=0}^n \frac{d_k^{(\gamma,\lambda)}}{\tau^{\gamma}}{\rm M_k} G_h^{n-k}=K_{\gamma}\,{\rm H}G_h^n+F^n.
\end{eqnarray}

\subsection{Stability and convergence}
Let $\mathcal{V}=\left\{v^n\big|v^n=\left(v^n_0,\cdots,v^n_m,\cdots,v^n_M\right), v^n_0=v^n_M=0\right\}$ be the grid function on the mesh
$\{(t_n,x_{0,m})\big|\,0\le n\le N, 0\le m\le M\}$.
For any $v,w\in \mathcal{V}$, we introduce the following notations
\begin{eqnarray}
&&(v^n,w^n)_h=h\sum_{m=1}^{M-1}v^n_m \overline {w^n_m},\quad \left\|v^n\right\|_h=\sqrt{(v^n,v^n)_h}\,,\\
&&\left|v^n\right|_{h,1}=\sqrt{h\sum_{m=1}^{M}\delta_{x_0}v^n_{m-\frac{1}{2} }\,\overline{\delta_{x_0}v^n_{m-\frac{1}{2} }}}, ~~~
 \left\|v^n\right\|_{h,\infty}=\max_{1\le m\le M-1}\left|v^n_m\right|,
\end{eqnarray}
where  $\delta_{x_0} v^n_{m-\frac{1}{2}}= \frac{v^n_m-v^n_{m-1}}{h}$. Then similar to the proof  of \cite[Lemma 1.1.1]{Sun:12}, there holds
\begin{eqnarray} \label{normeq1}
\left\|v^n\right\|_h\le \frac{b-a}{\sqrt{6}}\left|v^n\right|_{h,1},\quad \left\|v^n\right\|_{h,\infty}\le \frac{b-a}{\sqrt{2}}\left|v^n\right|_{h,1}.
\end{eqnarray}
According to (\ref{normeq1}),   the following    norms \cite{Gao:16}
\begin{eqnarray}
\left\||v|\right\|_{0,h,1}=\left[\sum_{n=1}^N \left|v^n\right|_{h,1}^2\tau\right]^{1/2}, \quad \left\||v|\right\|_{0^\prime,h,\infty}
=\sum_{n=1}^N \left|v^n\right|_{h,\infty}\tau
\end{eqnarray}
will be used in the analysis.
\begin{lemma}\label{lemmeqne3}
For any $L\in \mathbb{N}$, let   $Z=\left(Z^0,Z^1,\cdots,Z^L\right)^T$. Then
\begin{eqnarray}
\sum_{n=0}^L\sum_{k=0}^nd_k^{(\gamma,\lambda)}Z^{n-k}Z^n \ge 0.
\end{eqnarray}
 Here $Z^T$ denotes the transpose of $Z$.
\end{lemma}
\begin{proof}
By  the definition of $d_k^{(\gamma,\lambda)}$ and Lemma \ref{ceofficienteq}, one has
\begin{eqnarray}
&&d_0^{(\gamma,\lambda)}=1-e^{-\gamma \lambda \tau}(\lambda\tau)^\gamma=e^{-\gamma\lambda \tau}
\left((e^{\lambda\tau})^{\gamma}-(\lambda\tau)^\gamma\right)>0,\\
&& d^{(\gamma,\lambda)}_k\le 0 {~~\rm for~~} k\ge 1,\label{coeffixieq31}\\
&&\sum_{k=0}^{n}d^{(\gamma,\lambda)}_k> \left(1-e^{-\tau \lambda}\right)^{\gamma}-e^{-\gamma \lambda \tau}(\lambda\tau)^\gamma\nonumber\\
&&~~~~~~~~~~~~=e^{-\gamma\lambda \tau}\left( \left(e^{\lambda\tau}-1\right)^{\gamma}-   (\lambda\tau)^\gamma   \right)>0,
\end{eqnarray}
where the inequality $e^x>1+x, x>0$ has been used.
Define a lower triangular Toeplitz matrix ${\rm B}=(b_{k,j})_{L+1,L+1}$ with $b_{k,j}=d^{(\gamma,\lambda)}_{k-j}, j=0,\cdots, k$.
 Then matrix ${\rm B}$ is row and column  diagonally-dominant with $b_{k,k}>0$. Using the Gerschgorin theorem, one has
\begin{eqnarray*}
\sum_{n=0}^L\sum_{k=0}^nd_k^{(\gamma,\lambda)}Z^{n-k}Z^n=Z^T B Z=\frac{1}{2}Z^T\left({\rm B}+{\rm B}^T\right)Z \ge 0.
\end{eqnarray*}
The proof is complete.
\end{proof}

\begin{theorem}\label{stability2}
The discrete  scheme (\ref{differencescheme2}) is unconditionally stable, i.e., for any $1\le L\le N$, there holds
\begin{eqnarray}\label{stabilityeq22}
K_{\gamma} \sum_{n=1}^L\tau\left\|G_h^n\right\|_{h,\infty}^2\le \frac{\tau(b-a)^4}
{12K_\gamma}\sum_{n=1}^L\left\|F^n\right\|_h^2+\tau^{1-\gamma}(b-a)^2d_0^{(\gamma,\lambda)}\left\|G_h^0\right\|^2_{h}.
\end{eqnarray}
\end{theorem}

\begin{proof}
Taking an inner product of (\ref{differencescheme2}) with $G_h^n$, it yields that
\begin{eqnarray}\label{qnality}
\sum_{k=0}^n \frac{d_k^{(\gamma,\lambda)}}{\tau^{\gamma}}\left({\rm M_k} G_h^{n-k},G_h^n\right)_h
=K_{\gamma}\,\left({\rm H}G_h^n,G_h^n\right)_h+\left(F^n,G_h^n\right)_h.
\end{eqnarray}
Note that for ${\rm Re}(pU(x_0))\ge0$,  there holds ${\rm M_0}={\rm I}$ and
\begin{eqnarray*}
\left({\rm M_k} G_h^{n-k},G_h^n\right)_h\le \left\|{\rm M_k}G_h^{n-k}\right\|_h\left\|G_h^n\right\|_h\le \left\|G_h^{n-k}\right\|_h
\left\|G_h^n\right\|_h~~k\ge 1.
\end{eqnarray*}
Putting the above estimate into (\ref{qnality}), using (\ref{coeffixieq31}),
$ \left({\rm H}G_h^n,G_h^n\right)_h=-\left|G_h^n\right|^2_{h,1}$, and Young's inequality, one has
\begin{eqnarray}\label{lemmeqn22}
~~\sum_{k=0}^n \frac{d_k^{(\gamma,\lambda)}}{\tau^{\gamma}}\left\|G_h^{n-k}\right\|_h\left\|G_h^n\right\|_h+ K_{\gamma}\,\left|G_h^n\right|_{h,1}^2
\le {\epsilon}\left\|F^n\right\|_h^2+\frac{1}{4\epsilon}\left\|G_h^n\right\|^2_h.
\end{eqnarray}
Summing up (\ref{lemmeqn22}) for $n$ from $1$ to $L$, then adding $\tau^{1-\gamma}d_0^{(\gamma,s)}\left\|G_h^0\right\|_h^2$
 on both sides of  the obtained result and  using (\ref{normeq1}), one has
\begin{eqnarray}
&&{\tau^{1-\gamma}}\sum_{n=0}^L\sum_{k=0}^n {d_k^{(\gamma,\lambda)}}\left\|G_h^{n-k}\right\|_h\left\|G_h^n\right\|_h
+K_{\gamma} \sum_{n=1}^L\tau \left|G_h^n\right|_{h,1}^2\nonumber\\
&&\le \tau\sum_{n=1}^L\left( {\epsilon}\left\|F^n\right\|_h^2+\frac{1}{4\epsilon}\left\|G_h^n\right\|^2_h\right)
+\tau^{1-\gamma}d_0^{(\gamma,\lambda)}\left\|G_h^0\right\|^2_{h}\nonumber\\
&&\le \tau\sum_{n=1}^L\left( {\epsilon}\left\|F^n\right\|_h^2+\frac{1}{4\epsilon}\frac{(b-a)^2}{6}
\left|G_h^n\right|_{h,1}^2\right)+\tau^{1-\gamma}d_0^{(\gamma,\lambda)}\left\|G_h^0\right\|^2_{h}.
\label{eqneeee33}
\end{eqnarray}
Letting $Z=\left(\left\|G_h^0\right\|_h,\left\|G_h^1 \right\|_h,\cdots,\left\|G_h^L\right\|_h\right)^T$ and by Lemma \ref{lemmeqne3}, it holds that
\begin{eqnarray*}
\sum_{n=0}^L\sum_{k=0}^n {d_k^{(\gamma,\lambda)}}\left\|G_h^{n-k}\right\|_h\left\|G_h^n\right\|_h \ge 0.
\end{eqnarray*}
Putting the above result into (\ref{eqneeee33}) and choosing $\epsilon=\frac{(b-a)^2}{12 K_{\gamma}}$, one has
\begin{eqnarray}\label{normeqeuate}
K_{\gamma} \sum_{n=1}^L\tau\left|G_h^n\right|_{h,1}^2\le \frac{\tau(b-a)^2}{6K_\gamma}
\sum_{n=1}^L\left\|F^n\right\|_h^2+2\tau^{1-\gamma}d_0^{(\gamma,\lambda)}\left\|G_h^0\right\|^2_{h}.
\end{eqnarray}
Then  (\ref{stabilityeq22}) follows after using $\left\|G_h^n\right\|_{h,\infty}\le \frac{b-a}{\sqrt{2}}\left|G_h^n\right|_{h,1}$.
\end{proof}

\begin{theorem}
Let $G_{x_0}(p,t)$ be the exact solution of the  problem (\ref{modemode2}) and denote
\begin{eqnarray}
 \widetilde{G^n}=\left(\widetilde{G_1^n},\cdots, \widetilde{G_m^n},\cdots,\widetilde{G_{M-1}^n}\right)^T,\quad E_h^n=\widetilde{G^n}-G_h^n,\nonumber
 \end{eqnarray}
 where $\widetilde{G_m^n}:=G_{x_{0,m}}(p,t_n)$. Then
\begin{eqnarray}\label{eq34}
  \left\||E_h|\right\|_{0^\prime,h,\infty}<c_1\left(\tau+h^2\right),\quad \left\||E_h|\right\|_{0,h,1}<c_2\left(\tau+h^2\right),
\end{eqnarray}
where $c_1$ and $c_2$ are constants independent of $\tau$ and $h$, and their values can be simply got from Eqs. (38) and (39).
\end{theorem}
\begin{proof}
 By (\ref{timediscrete1}), (\ref{spacediscrete2}) and (\ref{differencescheme2}), the error equation can be given as
 \begin{eqnarray}\label{errorequation2}
\sum_{k=0}^n \frac{d_k^{(\gamma,\lambda)}}{\tau^{\gamma}}{\rm M_k} E_h^{n-k}=K_{\gamma}\,{\rm H}E_h^n+R_h^n,
\end{eqnarray}
where $E_h^0=0,\, R_h^n=\left(R_1^{n},\quad \cdots,R_m^n,\cdots,R_{M-1}^n\right)$,
 with $R_m^n=c_3(\tau+h^2), 1\le n\le N, 1\le m\le M-1$ being introduced in  (\ref{spacediscrete2}) and (\ref{timediscrete1}).
The application of Theorem \ref{stability2} to (\ref{errorequation2}) produces
\begin{eqnarray}
K_{\gamma} \sum_{n=1}^N\tau\left\|E_h^n\right\|_{h,\infty}^2\le \frac{\tau(b-a)^4}{12K_\gamma}
\sum_{n=1}^N\left\|R_h^n\right\|_h^2\le c_3^2T\frac{(b-a)^5}{12K_\gamma}(\tau+h^2)^2.
\end{eqnarray}
By the Cauchy-Schwarz inequality, it holds that
\begin{eqnarray} \label{eq38}
&& \left( \sum_{n=1}^N\tau\left\|E_h^n\right\|_{h,\infty}\right)^2\le\left(\sum_{n=1}^N\tau\right)\left(
 \sum_{n=1}^N\tau\left\|E_h^n\right\|_{h,\infty}^2\right)\nonumber\\
&&~~\le c_3^2T^2\frac{(b-a)^5}{12K_\gamma^2}(\tau+h^2)^2.\label{errorestimate1}
\end{eqnarray}
And from (\ref{normeqeuate}), it is easy to obtain
\begin{eqnarray} \label{eq39}
\left\||E_h|\right\|^2_{0,h,1}\le  c_3^2T\frac{(b-a)^3}{6K_\gamma^2}(\tau+h^2)^2.\label{errorestimate2}
\end{eqnarray}
The proof is complete.
\end{proof}
\section{Finite element scheme and the stability and convergence analysis}\label{sec4}
In this section, we further discuss the finite element approximation for (\ref{eq7}), including the corresponding stability and convergence analysis.
\subsection{Derivation of the finite element approximation}
As usual, we   seek the weak solution $G_{x_0}(p,\cdot)\in H^1_0(\Omega)$ of the model (\ref {modemode2})
 such that for any $\eta\in H_0^1(\Omega)$, there holds
\begin{eqnarray}
\left(D_t^{\gamma,\lambda}G_{x_0}(p,t),\eta\right)-\left(\lambda^\gamma G_{x_0}(p,t),\eta\right)
+K_\gamma\left(\frac{\partial G_{x_0}(p,t)}{\partial x_0},\frac{\partial \eta}{\partial x_0}\right)=\left(f,\eta\right),\label{weakformulae}
\end{eqnarray}
where  $\left(v,w\right)=\int_{\Omega}v\overline{w}dx_0$ denotes the complex inner product in $L_2(\Omega)$.
Let $a=x_{0,1}<\cdots <x_{0,m}<\cdots<x_{0,M}=b $ be a partition of $\Omega$ with $x_{0,m}=mh, h=\frac{(b-a)}{M}$.
 Then we define the finite element space:
\begin{eqnarray}
S_h:=\left\{v\in H_0^1(\Omega)\cap C(\overline {\Omega}): v\big|_{I_m}\in P_l(I_m)\right\},
\end{eqnarray}
where $P_l(I_m)$ denotes the space of polynomials of degree no greater than $l$ ($l\in\mathbb{N}^+$) on $I_m=(x_{0,m-1},x_{0,m})$.
Combining with (\ref{timediscrete1}),  the fully discrete finite element scheme can be described as: for $n=1,2,\cdots,N$, find $G_h^n\in S_h$ such that
\begin{eqnarray}\label{finteelement}
\left(\frac{1}{\tau^{\gamma}}\sum_{k=0}^n d_k^{(\gamma,\lambda)}e^{-pUk\tau}G_h^{n-k}, \,\eta\right)
+K_{\gamma}\left(\frac{\partial G_h^n}{\partial x_0}, \frac{\partial \eta}{\partial x_0}\right)=\left(f^n,\,\eta\right)~~ \forall\eta\in S_h
\end{eqnarray}
with $G_h^0=R_h\phi(p,x_0)=0\in S_h$, where $f^n=f(x_0,p,t_n)$.
\subsection{Stability and convergence}
For $v(x_0,t)$  defined on  $\overline{\Omega}\times [0,T]$,  we introduce the following notations:
\begin{eqnarray}
\left\|v\right\|_0=\sqrt{\left(v,v\right)},\quad\left|v\right|_1=\sqrt{\left(\frac{\partial v}{\partial{x_0}},\frac{\partial v}{\partial{x_0}}\right)},\quad \left\|v\right\|_{\infty,q}=\sup_{0\le t\le T}\left||v(\cdot,t)\right\|_q,
\end{eqnarray}
where $\left\|v\right\|_q$ denotes the norm in Sobolev space $H^q(\Omega) \,(H^0(\Omega):=L_2(\Omega))$. And for $v(a,t)=v(b,t)=0$,  it holds that
\begin{eqnarray}
\left\|v\right\|_0\le \frac{1}{\sqrt{\lambda_1}}\left|v\right|_1=\frac{b-a}{\pi}\left|v\right|_1,\label{normequality2}
\end{eqnarray}
where $\lambda_1$ denotes the smallest eigenvalue of operator $-\frac{\partial^2}{\partial {x_0}^2}$.
Then with the time partition $t_n, 0\le n\le N$ and $v^n:=v(\cdot,t_n)$, the following    norm \cite{Ervin:07}
\begin{eqnarray}
\left\||v|\right\|_{0,h,1}=\left[\sum_{n=1}^N\left|v^n\right|_1^2\tau\right]^{1/2}
\end{eqnarray}
will be used in the analysis.

For every $t$, define the Ritz projection  $R_h: H_0^1(\Omega)\to S_h$ by the orthogonal relation
\begin{eqnarray}
\left(\frac{\partial(R_hG_{x_0}(p,t))}{\partial x_0},\frac{\partial \eta}{\partial x_0}\right)=\left(\frac{\partial G_{x_0}(p,t)}{\partial x_0},\frac{\partial \eta}{\partial x_0}\right)\quad \forall \eta\in S_h.
\end{eqnarray}
 Then for $ G_{x_0}(p,t)\in H_0^1(\Omega)\cap H^q(\Omega)$, we have the well known approximation property \cite{Thomee:04}:
\begin{eqnarray}\label{approximationeq1}
\left\|\rho \right\|_0+h\left|\rho\right|_1\le c_4h^{\min\{l+1,q\}}\left\|G_{x_0}(p,t)\right\|_{q},
\end{eqnarray}
where $\rho=G_{x_0}(p,t)-R_hG_{x_0}(p,t)$.

\begin{theorem}
The finite element scheme (\ref{finteelement}) is unconditionally stable. And the approximate solution $G_h^n$ satisfies the following error estimate:
\begin{eqnarray}
\left\||E|\right\|_{0,h,1} \le c_5\left(\tau+h^{\min\{l,\,q-1\}}\right),
\end{eqnarray}
where $E=\left(G_{x_0}(p,t_1)-G_h^1,\cdots,G_{x_0}(p,t_n)-G_h^n,\cdots,G_{x_0}(p,t_N)-G_h^N\right)$ with $G_{x_0}(p,t_n)$
satisfying (\ref{weakformulae}), and  $c_5$ is constant independent of $ \tau$ and $h$, which can be obtained from Eq. (\ref{eq61}).
\end{theorem}
\begin{proof}
Taking $\eta=G^n_h$ in (\ref{finteelement}), it yields that
\begin{eqnarray}
\frac{1}{\tau^{\gamma}}\sum_{k=0}^n d_k^{(\gamma,\lambda)}\left(e^{-pUk\tau}G_h^{n-k}, \,G_h^n\right)
+K_{\gamma}\left(\frac{\partial G_h^n}{\partial x_0}, \frac{\partial G_h^n}{\partial x_0}\right)=\left(f^n,\,G_h^n\right).\label{stabilityeq}
\end{eqnarray}
Note that for ${\rm Re}\left(pU(x_0)\right)\ge 0$,
\begin{eqnarray}
\left(e^{-pUk\tau}G_h^{n-k}, \,G_h^n\right)\le \left\|G_h^{n-k}\right\|_0\left\|G_h^n\right\|_0.
\end{eqnarray}
Therefore,
\begin{eqnarray}
\frac{1}{\tau^{\gamma}}\sum_{k=0}^n d_k^{(\gamma,\lambda)}\left\|G_h^{n-k}\right\|_0\left\|G_h^n\right\|_0
+K_{\gamma}\left|G_h^n\right|_1^2\le {\epsilon}\left\|f^n\right\|_0^2+\frac{1}{4\epsilon}\left\|G_h^n\right\|_0^2. \label{finielemeteq3}
\end{eqnarray}
Summing up (\ref{finielemeteq3}) for $n$ from $1$ to $L$, then adding $\tau^{1-\gamma}d_0^{(\gamma,s)}
\left\|G_h^0\right\|_0^2$ on both sides of  the obtained result, and  using (\ref{normequality2}) and  Lemma \ref{lemmeqne3}, one has
 \begin{eqnarray}
 K_{\gamma}\sum_{n=1}^L\tau\left|G_h^n\right|_1^2\le \tau\sum_{n=1}^L\left(\epsilon\left\|f^n\right\|_0^2
 +\frac{(b-a)^2}{4\epsilon\pi^2}\left|G_h^n\right|_1^2\right)
 +\tau^{1-\gamma}d_0^{(\gamma,\lambda)}\left\|G_h^0\right\|_0^2.
 \end{eqnarray}
 Choosing $\epsilon=\frac{(b-a)^2}{2\pi^2K_\gamma}$, it yields that
 \begin{eqnarray}
K_{\gamma} \sum_{n=1}^L\tau\left|G_h^n\right|_{1}^2\le \frac{\tau(b-a)^2}{\pi^2 K_\gamma}
\sum_{n=1}^L\left\|f^n\right\|_0^2+2\tau^{1-\gamma}d_0^{(\gamma,\lambda)}\left\|G_h^0\right\|_0^2.\label{stabilityeq2}
\end{eqnarray}
So, the proof of the unconditional stability of the finite element scheme (\ref{finteelement}) is completed.

Now, we derive the error estimate. First, we decompose the error into two terms, i.e., $ G_{x_0}(p,t_n)-G_h^n=\rho^n+\theta^n$
with $\rho^n=G_{x_0}(p,t_n)-R_hG_{x_0}(p,t_n)$ and $\theta^n=R_hG_{x_0}(p,t_n)-G_h^n$. Then for any $\eta\in S_h$, it follows that
\begin{eqnarray}
\left(\frac{1}{\tau^{\gamma}}\sum_{k=0}^n d_k^{(\gamma,\lambda)}e^{-pUk\tau}\theta^{n-k}, \,\eta\right)
+K_{\gamma}\left(\frac{\partial \theta^n}{\partial x_0}, \frac{\partial \eta}{\partial {x_0}} \right)=\left(r_1^n+r^n_2,\,\eta\right),\label{errorequation}
\end{eqnarray}
where $r_1^n$ satisfying $\left\|r_1^n\right\|_0\le c_6\tau$ with $1\le n\le N$ is the time directional error term introduced in (\ref{timediscrete1}), and the space projection error 
\begin{eqnarray}
&&\left\|r_2^n\right\|_0=\Big\|-\frac{1}{\tau^{\gamma}}\sum_{k=0}^n d_k^{(\gamma,\lambda)}e^{-pUk\tau}\rho^{n-k}\Big\|_0\nonumber\\
&&\le\frac{1}{\tau^\gamma}\Big\| e^{-(\lambda+pU)t_n}\sum_{k=0}^n(Q^{\gamma}_k-Q^{\gamma}_{k-1})
e^{(\lambda+pU)t_{n-k}}\rho^{n-k}\Big\|_0+e^{-\gamma \lambda h}\lambda^{\gamma}\left\|\rho^n\right\|_0\nonumber\\
&&\le\frac{1}{\tau^\gamma}\Big\|e^{-(\lambda+pU)t_n}\sum_{k=0}^{n-1}Q_{n-1-k}^{\gamma}\int_{t_k}^{t_{k+1}}\frac{\partial}{\partial t}
\left(e^{(\lambda+pU)t}\rho \right)dt\Big\|_0+e^{-\gamma \lambda h}\lambda^{\gamma}\left\|\rho^n\right\|_0\nonumber\\
&&\le \frac{1}{\tau^\gamma}\sum_{k=0}^{n-1}Q_{n-1-k}^\gamma\int_{t_k}^{t_{k+1}}
\left\|\left(\lambda+pU+\frac{\partial}{\partial t}\right)\rho\right\|_0dt+e^{-\gamma \lambda h}\lambda^{\gamma}\left\|\rho^n\right\|_0\nonumber\\
&&\le\sum_{k=0}^{n-1}\frac{Q_{n-1-k}^\gamma}{ \tau^{\gamma-1}}\left(\left(\lambda+p\sup_{x_0\in \Omega} U(x_0)\right)\left\|\rho\right\|_{\infty,0}+\left\|\frac{\partial\rho}{\partial t}\right\|_{\infty,0}\right)
+\lambda^\gamma \left\|\rho^n\right\|_0.\label{errorfem1}
\end{eqnarray}
Here $Q_{-1}^\gamma=0,\,Q_k^{\gamma}=\sum_{j=0}^kg_j^{(\gamma,0)}>0$, and $\rho^0=0$ have been used in the third step.  Since
\begin{eqnarray}
(1-z)^\gamma=\sum_{k=0}^{\infty}g_k^{(\gamma,0)}z^k=Q_0^{\gamma}+\sum_{k=1}^{\infty}\left(Q_k^\gamma-Q_{k-1}^\gamma\right)z^k
=(1-z)\sum_{k=0}^{\infty}Q_k^{\gamma}z^k,
\end{eqnarray}
then $Q_k^{\gamma}$ are the coefficients of the power series of the function $(1-z)^{\gamma-1}$.
By the Stirling formula, it follows that \cite[p. 9]{Lubich:96}
\begin{eqnarray}
\sum_{k=0}^{n-1}(-1)^k   {\gamma-1\choose k}=(-1)^{n-1}{r-2\choose n-1}=\frac{n^{1-\gamma}}{\Gamma(2-\gamma)}+\mathcal{O}(n^{-\gamma}).
\end{eqnarray}
Hence
\begin{eqnarray}
\tau^{1-\gamma}\sum_{k=0}^{n-1}Q_{n-1-k}^\gamma=\frac{t_n^{1-\gamma}}{\Gamma(2-\gamma)}+\tau \mathcal{O}(t_n^{-\gamma})\le c_7 T^{1-\gamma}.\label{errorfem2}
\end{eqnarray}
Taking $\eta=\theta^n$ in (\ref{errorequation}) and  combining (\ref{approximationeq1}), (\ref{stabilityeq}), (\ref{stabilityeq2}),  (\ref{errorfem1}) and (\ref{errorfem2}),  one has
\begin{eqnarray}
&&K_{\gamma} \sum_{n=1}^N\tau\left|\theta^n\right|_{1}^2\le \frac{\tau(b-a)^2}{\pi^2 K_\gamma}
\sum_{n=1}^N\left\|r_1^n+r_2^n\right\|_0^2+2\tau^{1-\gamma}d_0^{(\gamma,\lambda)}\left\|\theta^0\right\|_0^2\nonumber\\
&&~~~~~~~~~~~~~~~~~~\le \frac{2(b-a)^2}{\pi^2 K_\gamma}T\left((b-a)c^2_6\tau^2+c^2_8h^{2\min\{l+1,q\}}\right),
\end{eqnarray}
where
\begin{eqnarray}
&&c_8=\left(c_4c_7T^{1-\gamma}\left(\lambda+p\sup_{x_0\in \Omega} U(x_0)\right)+c_4\lambda^\gamma\right)\left\|G_{x_0}(p,t)\right\|_{\infty,q}\nonumber\\
&&~~~~~~+c_4c_7T^{1-\gamma}\left\|\frac{\partial G_{x_0}(p,t)}{\partial t}\right\|_{\infty,q}.
\end{eqnarray}
By the triangle inequality and property (\ref{approximationeq1}), it yields
\begin{eqnarray} \label{eq61}
&& \left\||E|\right\|_{0,h,1}^2\le 2 \sum_{n=1}^N\tau\left|\theta^n\right|_{1}^2
+2 \sum_{n=1}^N\tau\left|\rho^n\right|_{1}^2\nonumber\\
&&~~~~~~~~~~~~~\le \frac{4(b-a)^2}{\pi^2 K_\gamma^2}T\left((b-a)c^2_6\tau^2+c^2_8h^{2\min\{l+1,q\}}\right)\nonumber\\
&&~~~~~~~~~~~~~~~~+2c^2_4\left\|G_{x_0}(p,t)\right\|_{\infty,q}^2 T h^{2\min\{l,q-1\}}.
\end{eqnarray}
The proof is complete.
\end{proof}

\section{ Numerical results}\label{sec5}
In this section, we present several examples to verify the theoretical results provided in the previous section.
In the examples, the parameters are taken as $\Omega=(0,1), K_r=1$ and $ U(x_0)=x_0$.

\begin{example}\label{example1}
In this example, we add a source term in model (\ref{modemode2}) such  that its exact solution is
\begin{eqnarray*}
G_{x_0}(p,t)=(t^2+1)e^{-(\lambda+px_0)t}\left(\sin(x_0)-x_0 sin(1)\right).
\end{eqnarray*}
The corresponding source term  and the initial  and  boundary conditions can be  derived from the exact solution.
\end{example}

By  formula  (\ref{huanyuan}), we can let
\begin{eqnarray}
G_{x_0}(p,t)=W_{x_0}(p,t)+\left(\sin(x_0)-x_0\sin(1)\right)e^{-(\lambda+px_0)t}.
\end{eqnarray}
Then $W_{x_0}(p,t)$ solves 
\begin{eqnarray}
&&e^{-\left(\lambda+px_0\right)t}\,{}_0^CD_t^{\gamma}\left(e^{\left(\lambda+px_0\right)t}W_{x_0}(p,t)\right)-\lambda^{\gamma}W_{x_0}(p,t)\nonumber\\
&&=K_{\gamma} \frac{\partial ^2}{\partial x_0^2}W_{x_0}(p,t)+g(x_0,p,t),
\end{eqnarray}
where
\begin{eqnarray*}
&&g(x_0,p,t)=e^{-(\lambda+px_0)t}\left(\sin(x_0)-x_0\sin(1)\right)\left(\frac{2}
{\Gamma(3-\gamma)}t^{2-\gamma}-\lambda^\gamma t^2\right)\nonumber\\
&&+K_\gamma t^2 e^{-(\lambda+px_0)t}\left[(p^2t^2-1)\sin(x_0)
-2pt\cos(x_0) + (2pt-p^2t^2x_0)\sin(1) \right].
\end{eqnarray*}

We  choose $l=1$  for  the finite element discretization.
Then at every time level $t=t_n$, both the differential matrixes of the finite difference  and the finite element schemes
are tridiagonal; the speedup method can be used  with the  cost $\mathcal{O}(M)$.
The numerical results are given in Tables \ref{tab:1-1}, \ref{tab:1-3}, and \ref{tab:1-5}, which well confirm the theoretical analysis. Removing the limitation of zero initial condition, instead of (\ref{eq16}), numerically one can use the discretization: 
\begin{eqnarray}
&&e^{-\left(\lambda+px_0\right)t}\,{}_0^CD_t^{\gamma}\left(e^{\left(\lambda+px_0\right)t}G_{x_0}(p,t)\right)\Big|_{t=t_n}\nonumber\\
&&=e^{-\left(\lambda+px_0\right)t}\,{}_0^CD_t^{\gamma}\left(e^{\left(\lambda+px_0\right)t}\left(G_{x_0}(p,t)
-e^{-\left(\lambda+px_0\right)t}G_{x_0}(p,0)\right)\right)\Big|_{t=t_n}\nonumber\\
&&=D_t^{\gamma,\lambda}\left(G_{x_0}(p,t)-e^{-(\lambda+px_0)t}\phi(p,x_0)\right)\Big|_{t=t_n}\nonumber\\
&&=\frac{1}{\tau^\gamma}\sum_{k=0}^ng_k^{(\gamma,\lambda)}e^{-px_0k\tau}G_{x_0}(p,t_{n-k})
-\frac{e^{-(\lambda+px_0)n\tau}}{\tau^\gamma} \sum_{k=0}^n{g_k^{(\gamma,0)}}\phi+\mathcal{O}(\tau).~~~~~~ \label{timediscrete2}
\end{eqnarray}
Using (\ref{timediscrete2}) in (\ref{timediscrete1}) leads to a new discretization scheme of (\ref{modemode2}).
The numerical results obtained with the corresponding finite difference and finite element schemes
are given in Tables \ref{tab:1-2}, \ref{tab:1-4}, and \ref{tab:1-6}, 
which show the same convergence order, but with slightly big numerical errors.

\begin{table}\fontsize{8.5pt}{11pt}\selectfont
\begin{center}
 \caption{Numerical results (the $\left\|\cdot\right\|_{0^{\prime},h,\infty}$ and $\left\|\cdot\right\|_{0,h,1}$ errors) of the finite difference scheme
 for Example \ref{example1} with $\lambda=3, T=1$, and  $h=1/2^{11}$.}
\begin{tabular}{cc|cc|cc|cc}
  \hline
  $ type $    & $\tau$   &\multicolumn{2}{c|}{$\gamma=0.3,p=1+i$ }  &\multicolumn{2}{c|}{$\gamma=0.5,p=5$ } & \multicolumn{2}{c}{$\gamma=0.8,p=10i$ }  \\
  \cline{3-8}
                                     &        & Err       & Rate             &Err          & rate               &Err           & Rate      \\
   \hline
                                     &$1/2^7$      & 1.1075e-06   & ---            & 1.2351e-06     &---             & 5.8745e-06     & ---       \\
$\left\|\cdot\right\|_{0,h,\infty}$  &$1/2^8$     &  5.5468e-07 &  0.9976          & 6.2154e-07    &0.9907          & 2.9626e-06     &0.9876     \\
                                     &$1/2^9$      &2.7751e-07   &   0.9991         & 3.1188e-07    & 0.9949         & 1.4869e-06    & 0.9946      \\
\hline
                                    &$1/2^7$      & 2.6618e-06  & ---             & 3.5930e-06     & ---            &1.6728e-05     & ---       \\
$\left\|\cdot\right\|_{0,h,1}$      &$1/2^8$      & 1.3343e-06  &  0.9964           & 1.8094e-06     &0.9897          & 8.4359e-06      & 0.9877    \\
                                    &$1/2^9$      & 6.6784e-07   &  0.9985          & 9.0812e-07    &0.9946          & 4.2324e-06     &0.9951      \\
\hline
\end{tabular} \label{tab:1-1}
\end{center}
\end{table}
\begin{table}\fontsize{8.5pt}{11pt}\selectfont
\begin{center}
 \caption{Numerical results (the $\left\|\cdot\right\|_{0^{\prime},h,\infty}$ and $\left\|\cdot\right\|_{0,h,1}$ errors, implemented with the time scheme
 (\ref{timediscrete2})) of the finite difference scheme
 for Example \ref{example1} with $\lambda=3, T=1$, and  $h=1/2^{11}$.}
\begin{tabular}{cc|cc|cc|cc}
  \hline
  $ type $    & $\tau$   &\multicolumn{2}{c|}{$\gamma=0.3,p=1+i$ }  &\multicolumn{2}{c|}{$\gamma=0.5,p=5$ } & \multicolumn{2}{c}{$\gamma=0.8,p=10i$ }  \\
  \cline{3-8}
                                     &        & Err       & Rate             &Err          & rate               &Err           & Rate      \\
   \hline
                                     &$1/2^7$      & 1.6347e-05   & ---            & 2.0377e-05     &---             &  6.9109e-05     & ---       \\
$\left\|\cdot\right\|_{0,h,\infty}$  &$1/2^8$     &  8.2423e-06   &0.9879          &  1.0326e-05     &0.9807         &  3.5002e-05     & 0.9814     \\
                                     &$1/2^9$      & 4.1381e-06   &0.9941         & 5.1977e-06      &  0.9903         & 1.7610e-05
      &0.9911      \\
\hline
                                    &$1/2^7$      & 4.6496e-05 & ---             &  6.8074e-05      & ---            &  1.9591e-04    & ---       \\
$\left\|\cdot\right\|_{0,h,1}$      &$1/2^8$      &2.3450e-05    &0.9875           & 3.4511e-05     &  0.9800        & 9.9249e-05     &  0.9811    \\
                                    &$1/2^9$       &1.1776e-05  &0.9938           & 1.7378e-05     &  0.9898       &  4.9944e-05     & 0.9907      \\
\hline
\end{tabular} \label{tab:1-2}
\end{center}
\end{table}

\begin{table}\fontsize{8.5pt}{11pt}\selectfont
\begin{center}
 \caption{Numerical results (the $\left\|\cdot\right\|_{0^{\prime},h,\infty}$ and $\left\|\cdot\right\|_{0,h,1}$ errors) of the finite difference scheme
 for Example \ref{example1} with $\lambda=3, T=1$, and $\tau=h^2$.}
\begin{tabular}{cc|cc|cc|cc}
  \hline
  $ type $    &   $h$   &\multicolumn{2}{c|}{$\gamma=0.3,p=1+i$ }  &\multicolumn{2}{c|}{$\gamma=0.5,p=5$ } & \multicolumn{2}{c}{$\gamma=0.8,p=10i$ }  \\
  \cline{3-8}
                                     &        & Err       & Rate             &Err          & rate               &Err           & Rate      \\
   \hline
                                     &$1/2^4$      & 2.9913e-06   & ---            &  6.6574e-06     &---             & 7.5610e-05     & ---       \\
$\left\|\cdot\right\|_{0,h,\infty}$  &$1/2^5$     &  7.4620e-07  & 2.0031          &  1.6625e-06     &  2.0016        &1.8583e-05      &2.0246     \\
                                     &$1/2^6$      & 1.8646e-07   & 2.0007         &  4.1552e-07      & 2.0004        & 4.6304e-06      &2.0048      \\
\hline
                                    &$1/2^4$      &  7.9241e-06  & ---             &  2.1450e-05      & ---            & 5.3436e-04     & ---       \\
$\left\|\cdot\right\|_{0,h,1}$      &$1/2^5$      & 1.9817e-06   &  1.9995            &  5.3784e-06     &1.9957       &  1.3292e-04       &  2.0073    \\
                                    &$1/2^6$       & 4.9547e-07   &  1.9999           & 1.3456e-06     &  1.9989        & 3.3188e-05     & 2.0018     \\
\hline
\end{tabular} \label{tab:1-3}
\end{center}
\end{table}
\begin{table}\fontsize{8.5pt}{11pt}\selectfont
\begin{center}
 \caption{Numerical results (the $\left\|\cdot\right\|_{0^{\prime},h,\infty}$ and $\left\|\cdot\right\|_{0,h,1}$ errors, implemented with the time scheme
 (\ref{timediscrete2})) of the finite difference scheme
 for Example \ref{example1} with $\lambda=3, T=1$, and $\tau=h^2$.}
\begin{tabular}{cc|cc|cc|cc}
  \hline
  $ type $    &   $h$   &\multicolumn{2}{c|}{$\gamma=0.3,p=1+i$ }  &\multicolumn{2}{c|}{$\gamma=0.5,p=5$ } & \multicolumn{2}{c}{$\gamma=0.8,p=10i$ }  \\
  \cline{3-8}
                                     &        & Err       & Rate             &Err          & rate               &Err           & Rate      \\
   \hline
                                     &$1/2^4$      & 1.1756e-05  & ---            &   2.3316e-05     &---           & 2.7890e-04      & ---       \\
$\left\|\cdot\right\|_{0,h,\infty}$  &$1/2^5$     &  2.9399e-06   & 1.9996          & 5.8392e-06     &1.9975        &6.8873e-05       &2.0178     \\
                                     &$1/2^6$      & 7.3505e-07   & 1.9998         &  1.4606e-06     & 1.9992        & 1.7179e-05      & 2.0033      \\
\hline
                                    &$1/2^4$      & 2.9254e-05  & ---             &   6.3071e-05     & ---            & 1.4813e-03     & ---       \\
$\left\|\cdot\right\|_{0,h,1}$      &$1/2^5$      & 7.3311e-06    & 1.9966           &  1.5825e-05   & 1.9947      &  3.6924e-04       &2.0042    \\
                                    &$1/2^6$       & 1.8339e-06   & 1.9991          & 3.9600e-06     & 1.9987        &9.2242e-05     & 2.0010    \\
\hline
\end{tabular} \label{tab:1-4}
\end{center}
\end{table}

\begin{table}\fontsize{8.5pt}{11pt}\selectfont
\begin{center}
 \caption{Numerical results (the $\left\|\cdot\right\|_{0,h,1}$-error) of the finite element scheme
  for Example \ref{example1} with $\lambda=3$ and $T=1$.}
\begin{tabular}{cc|cc|cc|cc}
  \hline
  $ \tau$      & $h$   &\multicolumn{2}{c|}{$\gamma=0.3,p=1+i$ }  &\multicolumn{2}{c|}{$\gamma=0.5,p=5$ } & \multicolumn{2}{c}{$\gamma=0.8,p=10i$ }  \\
  \cline{3-8}
                &        & Err       & Rate                         &Err          & rate                        &Err           & Rate      \\
   \hline
                &$1/2^4$    &3.1340e-04  & ---                       & 2.9617e-04     &---                       & 2.8873e-03    & ---       \\
$\tau=h$        &$1/2^5$    &1.5555e-04  &1.0106                     & 1.4700e-04     & 1.0106                   & 1.4190e-03     & 1.0248     \\
                &$1/2^6$    &7.7477e-05  & 1.0055                    & 7.3149e-05     &  1.0069                   &7.0208e-04      & 1.0151      \\
\hline
                &$1/2^4$    &3.0837e-04  & ---                       & 2.8818e-04      & ---                     & 2.7588e-03     & ---       \\
  $\tau=h^2$    &$1/2^5$    &1.5410e-04   & 1.0008                    & 1.4446e-04      &0.9963                   & 1.3844e-03     &0.9947   \\
                &$1/2^6$    &7.7037e-05   & 1.0002                    & 7.2275e-05     &0.9991                   & 6.9285e-04     &0.9987      \\
\hline
\end{tabular} \label{tab:1-5}
\end{center}
\end{table}
\begin{table}\fontsize{8.5pt}{11pt}\selectfont
\begin{center}
 \caption{Numerical results (the $\left\|\cdot\right\|_{0,h,1}$-error, implemented with the time scheme
 (\ref{timediscrete2})) of the finite element scheme
  for Example \ref{example1} with $\lambda=3$ and $T=1$.}
\begin{tabular}{cc|cc|cc|cc}
  \hline
  $ \tau$      & $h$   &\multicolumn{2}{c|}{$\gamma=0.3,p=1+i$ }  &\multicolumn{2}{c|}{$\gamma=0.5,p=5$ } & \multicolumn{2}{c}{$\gamma=0.8,p=10i$ }  \\
  \cline{3-8}
                &        & Err       & Rate                         &Err          & rate                        &Err           & Rate      \\
   \hline
                &$1/2^4$    & 3.1677e-03  & ---                       & 2.3810e-03     &---                       & 9.9236e-03    & ---       \\
$\tau=h$        &$1/2^5$    &1.6769e-03 & 0.9176                    & 1.2970e-03     & 0.8763                   & 4.9773e-03     & 0.9955     \\
                &$1/2^6$    & 8.6229e-04  & 0.9595                   & 6.7730e-04     & 0.9374                   &2.4903e-03     & 0.9990    \\
\hline
                &$1/2^4$    &3.5166e-03  & ---                       & 2.7857e-03      & ---                     &  9.8866e-03     & ---       \\
  $\tau=h^2$    &$1/2^5$    & 1.7676e-03   & 0.9924                   & 1.4047e-03    &0.9878                    & 4.9619e-03     & 0.9946    \\
                &$1/2^6$    & 8.8499e-04   & 0.9981                   & 7.0384e-04     & 0.9969                  &  2.4833e-03    &0.9986      \\
\hline
\end{tabular} \label{tab:1-6}
\end{center}
\end{table}

\begin{example}\label{example2}
In this example, we consider the  model (\ref{modemode2}) itself  with
\begin{eqnarray*}
\phi(p,x_0)=0,\quad \psi_l(p,t)=t,\quad \psi_r(p,t)=e^{- t}-1.
\end{eqnarray*}
\end{example}
By  formula  (\ref{huanyuan}), one can let
\begin{eqnarray*}
G_{x_0}(p,t)=W_{x_0}(p,t)+\left[(e^{-t}-1)e^{(\lambda+p)t}x_0+te^{\lambda t}(1-x_0)\right]e^{-(\lambda+px_0)t}
\end{eqnarray*}
to obtain a model of $W_{x_0}(p,t)$ with
\begin{eqnarray*}
&&g=-\left[x_0\,{}_0^CD_t^{\gamma}\left((e^{-t}-1)e^{(\lambda+p)t}\right)+(1-x_0)\,{}_0^CD_t^{\gamma}(te^{\lambda t})\right]e^{-(\lambda+px_0)t}\\
&&~~+\lambda^{\gamma}\left((e^{-t}-1)x_0e^{pt}+t(1-x_0)
-1\right)e^{-px_0t}+\lambda e^{-\left(\lambda+px_0\right)t}{}_0I_t^{1-\gamma}(e^{\lambda t})~~~~~\\
&&~~+K_\gamma\left[(e^{- t}-1)(p^2t^2x_0-2pt)e^{pt}+t^2(-p^2tx_0+p^2t+2p)\right]e^{-ptx_0},
\end{eqnarray*}
where for every time $t=t_n$,   $_0I_t^{1-\gamma}(e^{\lambda t})$ in $g(x_0,p,t)$ can be rewritten as
\begin{eqnarray}
&&_0I_t^{1-\gamma}(e^{\lambda t})\Big|_{t=t_n}=\frac{1}{\Gamma(1-\gamma)}\int_0^{t_n}(t_n-\xi)^{-\gamma}e^{\lambda \xi}d\xi\nonumber\\
&&~~~~~~~~~~~~~~~~~~~=\frac{t_n^{1-\gamma}}{2^{1-\gamma}\Gamma(1-\gamma)}\int_{-1}^1(1-\xi)^{-\gamma}e^{\frac{\lambda t_n(1+\xi)}{2}}d\xi.\label{Gaussintegral}
\end{eqnarray}
Then the Gauss Jacobi quadrature with the  weight $(1-\xi)^{-\gamma}$ can be used to compute (\ref{Gaussintegral}) \cite[Appendix A, p. 447]{Hesthaven:07}. 
Note that
\begin{eqnarray}\label{eq66}
&&{}_0^CD_t^{\gamma}\left((e^{- t}-1)e^{(\lambda+p)t}\right)={}_0I_t^{1-\gamma}\left(\frac{d}{dt}\left((e^{-t}-1)e^{(\lambda+p)t}\right)\right),\\\label{eq67}
&&{}_0^CD_t^{\gamma}(te^{\lambda t})={}_0I_t^{1-\gamma}\left(\frac{d}{dt}\left(te^{\lambda t}\right)\right);
\end{eqnarray}
so, Eqs. (\ref{eq66}) and (\ref{eq67}) can also be handled similar to (\ref{Gaussintegral}).
The numerical results are presented in Tables \ref{tab:2-1}-\ref{tab:2-3}, where
in the finite difference scheme, the numerical solutions obtained at  $ \tau=h=1/2^{12}$ have been regarded as the exact solutions; and in the finite element scheme, the $\left|\cdot\right|_1^*$-errors at $h$ denote $\left|G_{{h}/{2}}^N-G_{h}^N\right|_1$,
 which can be used to check the convergence order with $\left|\cdot\right|_1$ norm. In fact, there hold
 \begin{eqnarray*}
 &&\left|G_{{h}/{2}}^N-G_{h}^N\right|_1
 \le\left|G_{x_0}(p,t_N)-G_{{h}/{2}}^N\right|_1+\left|G_{x_0}(p,t_N)-G_{h}^N\right|_1\\
 &&~~~~~~~~~~~~~~~~~~\le 2 \left|G_{x_0}(p,t_N)-G_{h}^N\right|_1.
 \end{eqnarray*}

\begin{table}\fontsize{8.5pt}{11pt}\selectfont
\begin{center}
 \caption{Numerical results (the $\left\|\cdot\right\|_{h,\infty}$ and $\left|\cdot\right|_{h,1}$ errors)  of the finite difference scheme
 for Example \ref{example2}  with $p=5i,T=0.5$, and $h=1/2^{11}$.}
\begin{tabular}{cc|cc|cc|cc}
  \hline
  $ type  $                  & $\tau$   &\multicolumn{2}{c|}{$\gamma=0.3,\lambda=0$ }  &\multicolumn{2}{c|}{$\gamma=0.5,\lambda=3$ } & \multicolumn{2}{c}{$\gamma=0.8,\lambda=5$ }  \\
  \cline{3-8}                  &        & Err       & Rate                         &Err     & rate              &Err        & Rate                \\
   \hline
                                     &$1/2^7$    &2.1296e-04    & ---            & 6.4417e-04    &---           & 2.6421e-03       & ---         \\
 $\left\|\cdot\right\|_{h,\infty}$   &$1/2^8$    &  1.0480e-04   &  1.0229         & 3.1900e-04   &1.0139         & 1.3162e-03       & 1.0053   \\
                                     &$1/2^9$     &5.0722e-05   &  1.0470         & 1.5484e-04   & 1.0428        &6.4073e-04       & 1.0386     \\
\hline
                                    &$1/2^7$    & 4.7631e-04     & ---            &  1.4418e-03    & ---          & 5.9332e-03      & ---            \\
$\left|\cdot\right|_{h,1}$          &$1/2^8$    &2.3441e-04    &   1.0229        & 7.1397e-04   & 1.0139          &2.9557e-03       & 1.0053   \\
                                    &$1/2^9$    & 1.1345e-04     &  1.0470        & 3.4655e-04   & 1.0428         & 1.4389e-03      & 1.0386      \\
\hline
\end{tabular} \label{tab:2-1}
\end{center}
\end{table}
\begin{table}\fontsize{8.5pt}{11pt}\selectfont
\begin{center}
 \caption{Numerical results (the $\left\|\cdot\right\|_{h,\infty}$ and $\left|\cdot\right|_{h,1}$ errors)  of the finite difference scheme
 for Example \ref{example2}  with $p=5i,T=0.5$, and $\tau=h^2$.}
\begin{tabular}{cc|cc|cc|cc}
  \hline
  $ type  $                  & $h$   &\multicolumn{2}{c|}{$\gamma=0.3,\lambda=0$ }  &\multicolumn{2}{c|}{$\gamma=0.5,\lambda=3$ } & \multicolumn{2}{c}{$\gamma=0.8,\lambda=5$ }  \\
  \cline{3-8}                  &        & Err       & Rate                         &Err     & rate              &Err        & Rate                \\
   \hline
                                        &$1/2^4$    & 1.0019e-03    & ---            & 1.1577e-03     &---          & 1.6511e-03      & ---         \\
 $\left\|\cdot\right\|_{h,\infty}$     &$1/2^5$    & 2.4892e-04    &  2.0090         & 2.8575e-04   & 2.0184        & 3.8684e-04      &2.0937   \\
                                       &$1/2^6$     & 6.0962e-05    & 2.0297         & 6.8576e-05    &2.0590         &7.1566e-05      &2.4344    \\
\hline
                                        &$1/2^4$    &2.4312e-03     & ---            &2.7401e-03    & ---           & 3.7190e-03      & ---            \\
$\left|\cdot\right|_{h,1}$              &$1/2^5$     &6.0597e-04    &  2.0043        &6.7931e-04    &   2.0121      & 8.7433e-04      &2.0887   \\
                                        &$1/2^6$    & 1.4917e-04    &  2.0223        & 1.6489e-04    &  2.0426      &  1.6538e-04     &2.4024      \\
\hline
\end{tabular} \label{tab:2-2}
\end{center}
\end{table}
\begin{table}\fontsize{8.5pt}{11pt}\selectfont
\begin{center}
 \caption{Numerical results ($\left|\cdot\right|^*_{1}$-error) of the finite element scheme
   for Example \ref{example2}  with $p=5i$, and $T=0.5$.}
\begin{tabular}{cc|cc|cc|cc}
  \hline
  $ \tau $         & $h$   &\multicolumn{2}{c|}{$\gamma=0.3,\lambda=0$ }  &\multicolumn{2}{c|}{$\gamma=0.5,\lambda=3$ } & \multicolumn{2}{c}{$\gamma=0.8,\lambda=5$ }  \\
  \cline{3-8}       &        & Err       & Rate                        &Err     & rate              &Err        & Rate                \\
   \hline
                   &$1/2^4$    & 6.8170e-02  & ---                       &   7.2097e-02    &---          & 8.0231e-02        & ---         \\
 $\tau=h$          &$1/2^5$    &  3.4093e-02   & 0.9997                   & 3.6046e-02   &  1.0001      & 4.0452e-02        &  0.9879   \\
                   &$1/2^6$     &1.7047e-02   &  1.0000                   & 1.8021e-02    & 1.0002       &2.0323e-02       & 0.9931     \\
\hline
                   &$1/2^4$    & 6.8112e-02   & ---                       &  7.1635e-02    & ---          &  7.5383e-02      & ---            \\
$\tau=h^2$         &$1/2^5$    & 3.4072e-02    & 0.9993                 & 3.5832e-02      &  0.9994      & 3.7672e-02      & 1.0008  \\
                   &$1/2^6$    & 1.7038e-02    & 0.9998                  & 1.7918e-02    &  0.9998      & 1.8833e-02      &1.0002      \\
\hline
\end{tabular} \label{tab:2-3}
\end{center}
\end{table}
\begin{figure}
\begin{center}
\includegraphics[width=2.2in,height=1.3in,angle=0]{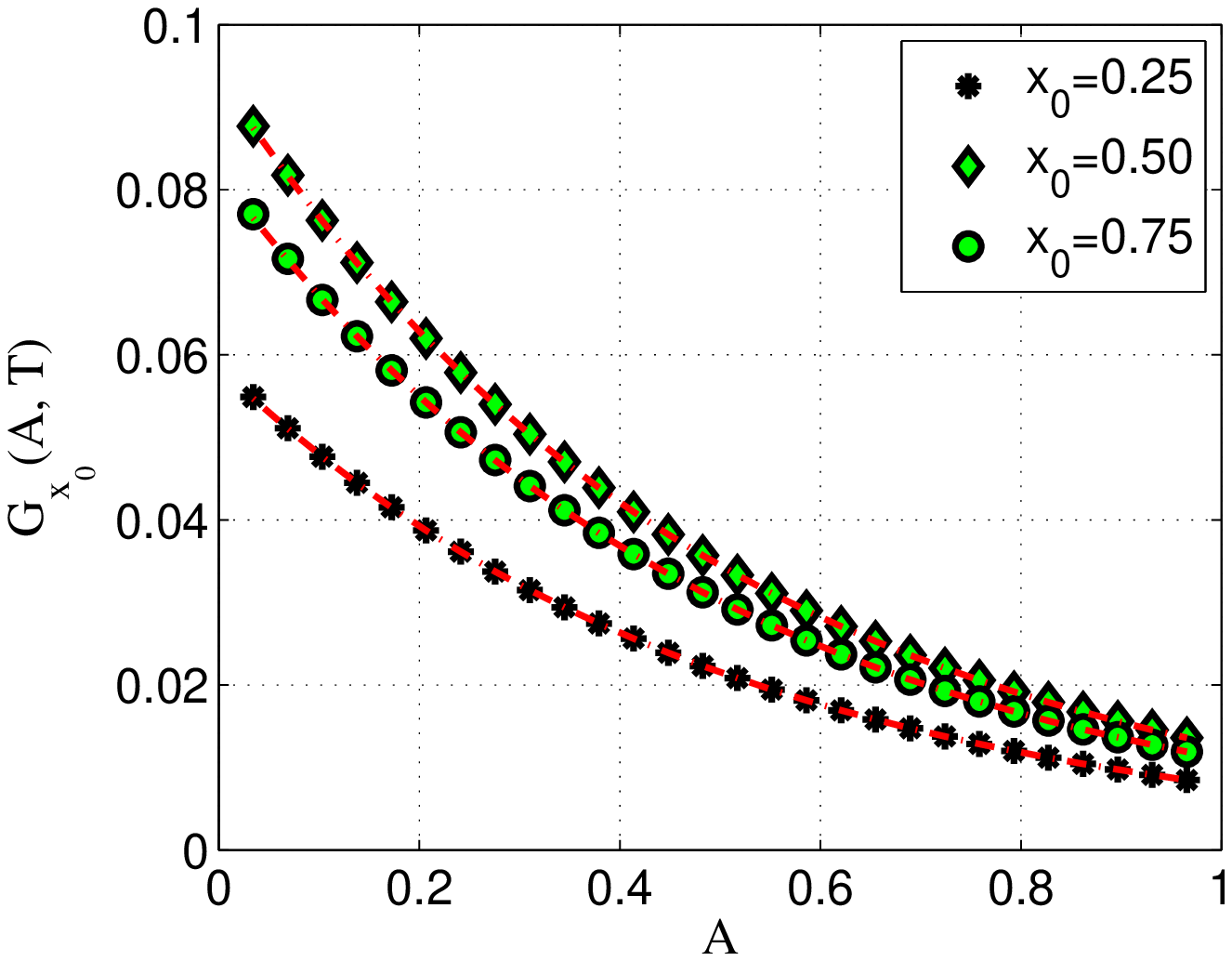}
\includegraphics[width=2.2in,height=1.3in,angle=0]{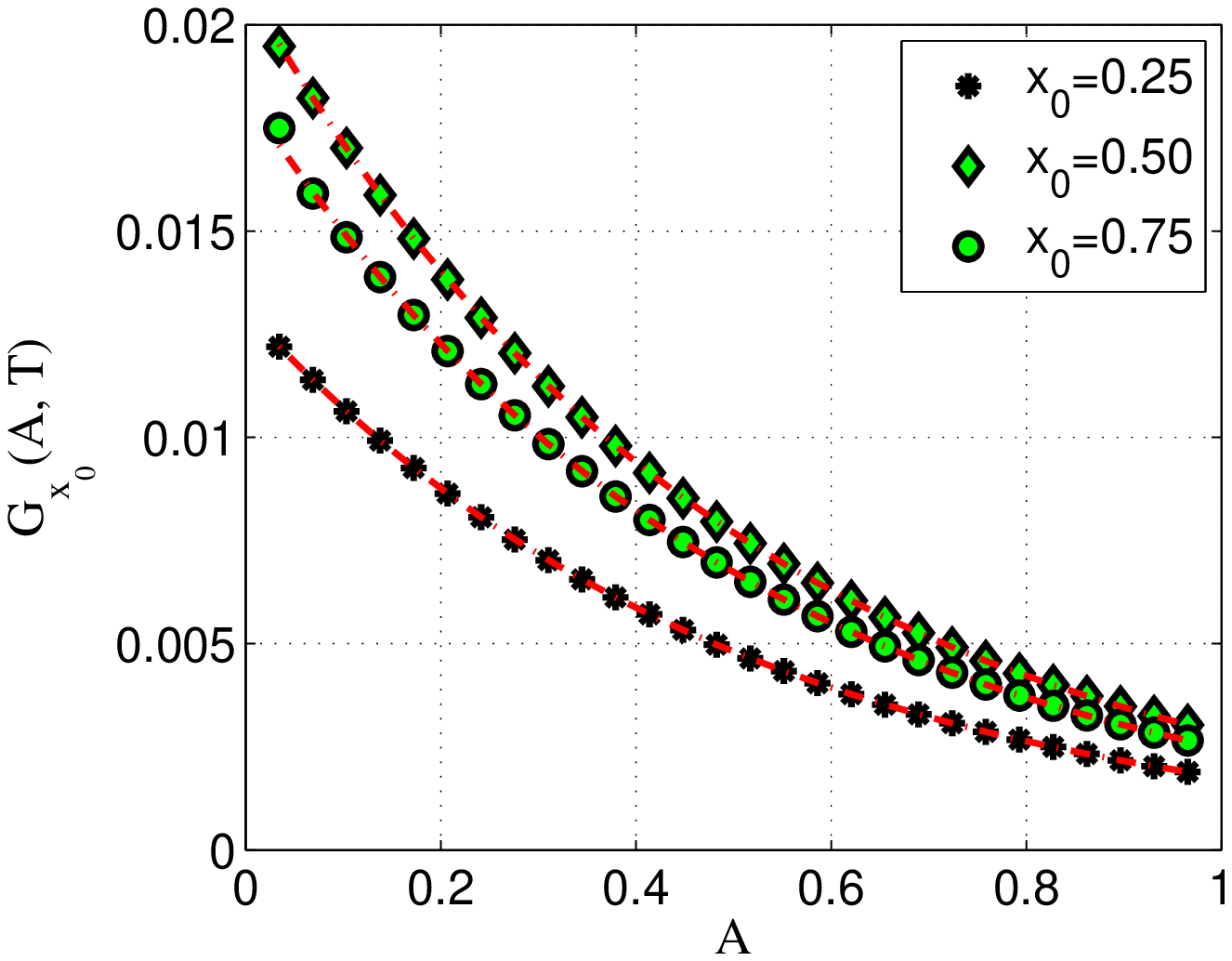}
\caption{Simulation results of $G_{x_0}(A,t)$ (left: $\lambda=0,\gamma=0.3$, and right: $\lambda=3,\gamma=0.6$) for Example (\ref{example3})
with  $h=\tau=1/2^{10}, K_1=25, K_2=15, \tilde{A}=18.4$, and $T=0.5$; (red) real lines are for analytical solutions.}\label{fig3-1}
\end{center}
\end{figure}

\begin{example}\label{example3}
If the {\rm PDF} $G_{x_0}(A,t)$ is needed, the numerical inversion of the Laplace transform (NIL) must be used. Let the exact solution of (\ref{modemode2}) be
\begin{eqnarray}\label{simulationeq1}
 G_{x_0}(p,t)=t^2 e^{-\lambda t}(x_0-x_0^3)/(p+2).
\end{eqnarray}
Then the corresponding right hand term $f(x_0,p,t)$ can be derived from the exact solution.
\end{example}

We use the Fourier series Euler summation method proposed in \cite{Abate:95} to simulate $G_{x_0}(A,t)$:
\begin{eqnarray*}
&&G_{x_0}(A,t)=\frac{1}{2\pi i}\int_{\sigma-\infty i}^{\sigma+\infty i}e^{pA} G_{x_0}(p,t)dp
\approx \sum_{k=0}^{K_2}{K_2\choose k}2^{-K_2}\,s_{K_1+k}(A),
\end{eqnarray*}
where  $i=\sqrt{-1}$ and $s_{n}(A)$ is the $n$-th partial sum:
\begin{eqnarray*}
s_n(A)= \frac{e^{\tilde{A}/2}}{2A}{\rm Re}\left(G_{x_0}\left(\frac{\tilde{A}}{2A},t\right)\right)
+\frac{e^{\tilde{A}/2}}{A}\sum_{j=1}^{n}(-1)^j {\rm Re}\left(G_{x_0}\left(\frac{\tilde{A}}{2A}+\frac{j\pi i}{A},t\right)\right).
\end{eqnarray*}
In practice, Abate and Whitt \cite{Abate:95} suggest to set $\tilde{A}$ equal to $18.4$.
The numerical results are plotted in Figure \ref{fig3-1},
being compared with the analytical inversion of $G_{x_0}(A,t)=t^2 e^{-\lambda t-2A}(x_0-x_0^3)$
(denoted by (red) real lines). 
For fixed $x_0$ and $t$, the finite difference scheme (\ref{differencescheme2})
has been  used to compute $G_{x_0}(p_j,t)$. Here, $ p_j=\frac{\tilde{A}}{2A}+\frac{j\pi i}{A}$ and $ j=0,\cdots,K_1+K_2$.

\begin{remark}
Though there have been a lot of works on the NIL (see, e.g., \cite{Cohen:10,lseger:06} and the references therein),
to ensure the validity and stability of the inversion, besides the values $G_{x_0}(p_j,t)$ the further requirements on $G_{x_0}(p,t)$ (even on $G_{x_0}(A,t)$, being hard to know in advance)
 and the multi-precision computation are usually needed. 
 We will discuss these in our future work.
\end{remark}

\section{Conclusions}
Based on the presented discretization of the tempered fractional substantial derivative, we have proposed the finite difference and finite element methods for the backward time tempered fractional Feynman-Kac equation with the detailed unconditional stability and convergence analyses. Some important ideas/techniques for the discretization of the tempered fractional substantial derivative and the proof of the stability and convergence are introduced. The effectiveness of the schemes, including unconditional stability and the order of convergence, is verified by the numerical experiments.

\def\ack{\section*{Acknowledgements}%
  \addtocontents{toc}{\protect\vspace{6pt}}%
  \addcontentsline{toc}{section}{Acknowledgements}}
\ack{This work was supported by the National Natural Science Foundation of China under Grant  No. 11271173.}


\appendix

\section{Proof of Eq. (\ref{IdenttityOperator2}):}\label{appendix1}
First,  ${}_0D_t^{1-\gamma}{}_0I_t^{1-\gamma}v(t)=v(t)$ is a direct result of \cite[Property 1.18]{Zhou:14}.
Note that for $v(t)\in C^1[0,t]$, one has \cite[Eq. (2.113), P. 71]{Podlubny:99}
\begin{eqnarray*}
{}_0I_t^{\gamma}\,{}_0D_t^{\gamma}v(t)=v(t)-\left[{}_0I_t^{1-\gamma}v(t)\right]_{t=0}\frac{t^{\alpha-1}}{\Gamma(\alpha)}.
\end{eqnarray*}
And using the integration by part, it holds that
\begin{eqnarray*}
{}_0I_t^{1-\gamma}v(t)=\frac{t^{1-\gamma}v(0)}{\Gamma(2-\gamma)}+\frac{1}{\Gamma(2-\gamma)}\int_0^t(t-\xi)^{1-\gamma}v^{\prime}(\xi)d\xi.
\end{eqnarray*}
Then (\ref{IdenttityOperator2}) follows from $\left[{}_0I_t^{1-\gamma}v(t)\right]_{t=0}=0$.

\section{Proof of Eq. (\ref{eq6}):}\label{appendix2}
Noting that
\begin{eqnarray}
&&\frac{\partial }{\partial t}\left( e^{-(\lambda+pU(x_0))t}{}_0I_t^{\gamma}\left(e^{(\lambda+pU(x_0))t}G_{x_0}(p,t)\right)\right)\nonumber\\
&&~~=-(\lambda+pU(x_0))\left( e^{-(\lambda+pU(x_0))t}{}_0I_t^{\gamma}\left(e^{(\lambda+pU(x_0))t}G_{x_0}(p,t)\right)\right)\nonumber\\
&&~~~~~+e^{-(\lambda+pU(x_0))t}\frac{\partial }{\partial t}\left({}_0I_t^{\gamma}\left(e^{(\lambda+pU(x_0))t}G_{x_0}(p,t)\right)\right),\nonumber
\end{eqnarray}
then using the definitions of the Riemann-Liouville derivative and the Caputo derivative, and (\ref{substantialdefinition}), one arrives the desired result.


\end{document}